\documentclass[12pt,oneside]{article}

\usepackage{quotes}
\usepackage{amssymb}
\usepackage{amsmath}
\usepackage{tikz}
\usetikzlibrary{shapes,arrows,positioning}
\usepackage[letterpaper,left=1in,right=1in, bottom=1in, top=1in]{geometry}
\usepackage{titlesec}
\usepackage{setspace}
\usepackage{textcomp}
\usepackage{amsthm}
\usepackage{eso-pic}
\usepackage{graphicx}
\usepackage{color}
\usepackage{longtable}
\usepackage{transparent}
\usepackage{caption}
\usepackage{hyperref}

\renewcommand{\geq}{\geqslant}
\renewcommand{\leq}{\leqslant}

\newcommand{\tab}{\hspace{1cm}}

\newcommand{\ind}{\text{\normalfont{ind}}}

\titleformat{\section}{\bfseries\centering}{\thesection \hspace{0.5cm}}{12pt}{}

\titleformat{\subsection}{\bfseries}{\thesubsection \hspace{0.5cm}}{12pt}{}

\newtheorem{theorem}{Theorem}[section]
\newtheorem{lemma}{Lemma}[section]
\newtheorem{proposition}{Proposition}[section]
\newtheorem{corollary}{Corollary}[section]
\newtheorem{remark}{Remark}[section]
\newtheorem{definition}{Definition}[section]

\newtheorem{property}{Property}[section]

\begin{document}\setlength{\parindent}{0cm}\thispagestyle{empty}
	\begin{center}
		\Large \textbf{On Arithmetic Cordial Labeling of Product Graphs}
	\end{center}
	\vspace{2mm}
	\begin{center}
		
	Jason D. Andoyo
	
	\end{center}
	
	\section*{Abstract}\small
	\tab Let $\eta$ be a fixed positive integer. Let $S$ be a subset of $\mathbb{Z}$, $\star:S\times S\to \mathbb{Z}$ be a binary function, and $\zeta_{\eta}:\{\xi\in \mathbb{Z}:\gcd(\xi,\eta)=1\}\to \{0,1\}$ be a function. For a simple graph $G$ of order $n$, a bijective function $f:V(G)\to S$ (where $|S|=n$) is called an arithmetic cordial labeling modulo $\eta$ under the arithmetic structure $\langle S,\zeta_\eta,\star\rangle$ if the induced function $f_\eta^*:E(G)\to \{0,1\}$, defined by $f_\eta^*(ab)=1$ whenever $\gcd(f(a)\star f(b),\eta)= 1$ and $\zeta_\eta(f(a)\star f(b))=1$; otherwise, $f_\eta^*(ab)=0$, satisfies the condition $|e_{f_\eta^*}(0)-e_{f_\eta^*}(1)|\leq 1$, where $e_{f_\eta^*}(i)$ is the number of edges with label $i$ ($i=0,1$). In this paper, the arithmetic cordial labeling of product graphs, namely, corona, lexicographic, cartesian, tensor, and strong, is explored under the operation of addition.
	\normalsize
	\vspace{0.3cm}
	
	\textbf{Keywords:} binary function; arithmetic cordial labeling; arithmetic structure; product graphs
	
	\textbf{MSC 2020:} 05C76, 05C78, 11A05, 11A07, 11A25.
	
	\section{Introduction}
	
	A simple graph $G$ is defined as an ordered pair $(V,E)$, where $V$ and $E$ are called the vertex set and edge set, respectively, such that $E$ is the set of unordered pairs of distinct elements of $V$. The elements of $V$ and $E$ are called vertices and edges, respectively, and the cardinalities $|V|$ and $|E|$ are called the order and size, respectively. If $v\in V(G)$, the degree of $v$, denoted by $\deg(v)$, is the number of vertices of $G$ that are adjacent to $v$. If $\deg(v)=1$, then we call $v$ a pendant vertex of $G$. 
	
	Graph labeling is a central concept in graph theory that involves assigning labels—such as numbers, sets, or any well-defined mathematical objects--to the vertices, edges, or both, subject to specific conditions \cite{Gallian}. A well-known variation is cordial labeling, which was first introduced by I. Cahit \cite{Cahit}. This foundational concept has inspired numerous variants, including those explored in \cite{Andoyo1,Andoyo2,Prudencio,Andoyo3,Andoyo4}. 
	
	Ultimately, these studies led to the introduction of arithmetic cordial labeling in \cite{Andoyo5}. While this initial framework established foundational properties and demonstrated the existence of this labeling for several graphs--including corona graphs and the tensor product $K_p \times G$ where $K_p$ is a complete graph of order $p$ and $G$ is a bipartite graph--a systematic investigation into the general form of graph products under the operation of addition remained open. In this study, we establish conditions under which various fundamental product graphs—namely, the corona, lexicographic, cartesian, tensor, and strong products—admit arithmetic cordial labeling under the operation of addition.
	Furthermore, we extend the definition of arithmetic cordial labeling to include all simple graphs, removing the previous restriction that the graph must be connected. This extension is necessary because graph operations like the tensor product can produce disconnected graphs even when the starting graphs are connected. By using this general definition, we can analyze both connected and disconnected product graphs under a single framework.
	
	Suppose that $\eta$ is a fixed positive integer and $G$ is a simple graph with $n$ vertices. Let $S$ be a subset of $\mathbb{Z}$ with $|S|=n$, $\star:S\times S\to \mathbb{Z}$ be a binary function, and $\zeta_\eta:\{\xi: \gcd(\xi,\eta)=1\}\to\{0,1\}$ be a function. The triple $\langle S,\zeta_\eta,\star\rangle$ is called \textit{arithmetic structure}. A bijective function $f:V(G)\to S$ is called \textit{arithmetic cordial labeling modulo $\eta$ under $\langle S,\zeta_\eta,\star\rangle$} whenever the induced function $f_\eta^*:E(G)\to \{0,1\}$, defined by
	$$f_\eta^*(uv)=\begin{cases}
		1&\text{ if }\gcd(f(u)\star f(v),\eta)=1\text{ and }\zeta_\eta(f(u)\star f(v))=1\\
		0&\text{ otherwise},
	\end{cases}$$
	satisfies the condition $|e_{f_\eta^*}(0)-e_{f_\eta^*}(1)|\leq 1$, where $e_{f_\eta^*}(i)$ is the number of edges with label $i$ ($i=0,1$). If a graph admits this labeling, we call it \textit{arithmetic cordial graph modulo $\eta$ under $\langle S,\zeta_\eta,\star\rangle$}. Furthermore, the function $f$ is called a \textit{perfect (positive) (negative) arithmetic cordial labeling modulo $\eta$ under $\langle S,\zeta_\eta,\star\rangle$} whenever $e_{f_\eta^*}(0)-e_{f_\eta^*}(1)=0$ ($e_{f_\eta^*}(0)-e_{f_\eta^*}(1)=1$) ($e_{f_\eta^*}(0)-e_{f_\eta^*}(1)=-1$). If a graph admits this labeling, then the graph is called \textit{perfect (positive) (negative) arithmetic cordial graph modulo $\eta$ under $\langle S,\zeta_\eta,\star\rangle$}. 
	\section{Basic Concepts}
	
	\begin{definition}
		A path graph $P_n$ of order $n$ and size $n-1$ is a graph with vertex set $V(P_n)=\{v_1,v_2,\ldots,v_n\}$ and edge set $E(P_n)=\{v_1v_2,v_2v_3,\ldots,v_{n-1}v_n\}$. Moreover, a cycle graph $C_n$ of order $n$ and size $n$ is obtained from path graph $P_n$ with an additional edge $v_1v_n$
	\end{definition}
	
	\begin{definition}
		A tadpole graph $T_{n,m}$ of order $n+m$ and size $n+m$ is obtained from a cycle graph $C_n$ and a path graph $P_{m+1}$ such that a pendant vertex of $P_{m+1}$ is a vertex of $C_n$.
	\end{definition}
	
	\begin{definition}
		Let $G$ and $H$ be graphs.
		\begin{enumerate}
			\item[i.] A join graph $G+H$ is a graph with vertex set $V(G+H)=V(G)\cup V(H)$ and edge set $E(G+H)=E(G)\cup E(H)\cup \{uv:u\in V(G)\text{ and }v\in V(H)\}$. 
			\item[ii.] Suppose that $G$ has order $n$. The corona graph $G\circ H$ is obtained by taking one copy of $G$ and $n$ copies of $H$ such that all vertices of the $i$th copy of $H$ are adjacent to $i$th vertex of $G$.
			\item[iii.] A lexicographic product graph $G\otimes H$ is a graph with vertex set $V(G\otimes H)=V(G)\times V(H)$ and edge set
			$$E(G\otimes H)=\{(v_1,u_1)(v_2,u_2):v_1v_2\in E(G),\text{ or }v_1=v_2\text{ and }u_1u_2\in E(H)\}.$$
			\item[iv.] A cartesian product graph $G\square H$ is a graph with vertex set $V(G\square H)=V(G)\times V(H)$ and edge set
			$$E(G\square H)=\{(v_1,u_1)(v_2,u_2):v_1=v_2\text{ and }u_1u_2\in E(H),\text{ or }u_1=u_2\text{ and }v_1v_2\in E(G)\}.$$
			\item[v.] A tensor product graph $G\times H$ is a graph with vertex set $V(G\times H)=V(G)\times V(H)$ and edge set
			$$E(G\times H)=\{(v_1,u_1)(v_2,u_2):v_1v_2\in E(G)\text{ and }u_1u_2\in E(H)\}.$$
			\item[vi.] A strong product graph $G\boxtimes H$ is a graph with vertex set $V(G\boxtimes H)=V(G)\times V(H)$ and edge set
			$$E(G\boxtimes H)=E(G\square H)\cup E(G\times H).$$
		\end{enumerate}
	\end{definition}

	\section{Main Results}
	Let $\Phi(\eta)=\{\xi:\gcd(\xi,\eta)=1\text{ and }0<\xi<\eta\}$. So, $|\Phi(\eta)|=\phi(\eta)$ where $\phi$ is the Euler phi function. Note that there are $\eta-\phi(\eta)$ integers in the set $\{0,1,\ldots,\eta-1\}$ that are not relatively prime to $\eta$. In addition, let $A_j=\{a\in \Phi(\eta):\zeta_\eta(a)=j\}$ for $j=0,1$. Moreover, in this section, we assume that $\zeta_\eta$ satisfies the following property.
	
	\begin{property}\label{pro1}
		Suppose that $\gcd(\theta_1,\eta)=\gcd(\theta_2,\eta)=1$. If $\theta_1\equiv \theta_2\pmod{\eta}$, then $\zeta_\eta(\theta_1)=\zeta_\eta(\theta_2)$.
	\end{property}
	
	In the following results, let $H$ be a graph of order $k\eta$, where $k,\eta\geq 1$, and suppose that $h:V(H)\to \{1,2,\ldots,k\eta\}$ is a bijective function. We partition the range of $h$, which is the set $\{1,2,\ldots,k\eta\}$, into subsets
	$$\lambda_t=\{t\eta+i:i=1,2,\ldots,\eta\}$$
	for $t=0,1,\ldots,k-1$ and define 
	\begin{equation}
		D_t=\{u\in V(H):h(u)\in \lambda_t\}. \label{partition}
	\end{equation}
	
	Let $R=\{[1],[2],\ldots,[\eta]\}$ be the set of residue classes modulo $\eta$. Obviously, $[0]=[\eta]$. Note that $[i]=\{t\eta+i:t\in \mathbb{Z}\}$ for $i=1,2,\ldots,\eta$. Clearly, for the elements of $\lambda_t$, $t\eta+i$ only belongs to $[i]$ for $i=1,2,\ldots,\eta$ and so $\lambda_t$ has one-to-one correspondence with $R$. Because $h$ is a bijective function, it follows that $D_t$ has one-to-one correspondence with $\lambda_t$. So, $D_t$ and $R$ have one-to-one correspondence too. This implies that images of the vertices in each set $D_t$ under $h$ contain exactly one representative of each residue class modulo $\eta$. We state this in the following remark.
	
	\begin{remark}\label{Dt}
		Let $R=\{[1],[2],\ldots,[\eta]\}$ be the set of residue classes modulo $\eta$. Then the sets $D_t$ and $R$ have one-to-one correspondence, for each $t=0,1,\ldots,k-1$, with respect to $H$.
	\end{remark}
	
	\begin{theorem}\label{thm1}
		Let $\eta\geq 3$ be an integer and let $G$ be a graph of order $n$. In addition, let $g:V(G)\to \{1,2,\ldots,n\}$ be a bijective function. Suppose further that
		\begin{align}	
			\Gamma_0^G=&\{uv\in E(G):\gcd(g(u)+g(v),\eta)\neq 1\}\nonumber\\
			&\cup \{uv\in E(G):\gcd(g(u)+g(v),\eta)= 1\text{ and }\zeta_\eta(g(u)+g(v))=0\},\label{eqcor1}\\
			\Gamma_1^G=&\{uv\in E(G):\gcd(g(u)+g(v),\eta)=1\text{ and }\zeta_\eta(g(u)+g(v))=1\},\label{eqcor2}\\
			\Gamma_0^H=&\{uv\in E(H):\gcd(h(u)+h(v),\eta)\neq 1\}\nonumber\\
			&\cup \{uv\in E(H):\gcd(h(u)+h(v),\eta)= 1\text{ and }\zeta_\eta(h(u)+h(v))=0\},\label{eqcor3}\\
			\Gamma_1^H=&\{uv\in E(H):\gcd(h(u)+h(v),\eta)=1\text{ and }\zeta_\eta(h(u)+h(v))=1\}.\label{eqcor4}
		\end{align} 
		Now, let 
		\begin{equation}
			\varpi(G\circ H;g,h)=nk[\eta-\phi(\eta)+|A_0|-|A_1|]+n[|\Gamma_0^H|-|\Gamma_1^H|]+|\Gamma_0^G|-|\Gamma_1^G|. \label{varpicor}
		\end{equation}
		If the functions $g$ and $h$ satisfy
		\begin{equation}
			\left|\varpi(G\circ H;g,h)\right|\leq 1,\label{eqcor5}
		\end{equation} 
		then the corona product graph $G\circ H$ is an arithmetic cordial graph modulo $\eta$ under $\langle S,\zeta_\eta,+\rangle$, 
		where $S=\{1,2,\ldots,n(k\eta+1)\}$. In addition, if $\varpi(G\circ H;g,h)$ is equal to $0$, $1$, or $-1$, then $G\circ H$ is a perfect, positive, or negative arithmetic cordial graph modulo $\eta$ under $\langle S,\zeta_\eta,+\rangle$, respectively.
	\end{theorem}

	\begin{proof}
		Let $V(G)=\{v_1,v_2,\ldots,v_n\}$ and $V(H)=\{u_1,u_2,\ldots,u_{k\eta}\}$. Suppose that $V(H^i)=\{u_1^i,u_2^i,\ldots,u_{k\eta}^i\}$ is the vertex of the $i$th copy of $H$ for which every vertex of $H^i$ is adjacent to $v_i$, for $i=1,2,\ldots,n$. Now, let $f:V(G\circ H)\to S$ be a function defined by
		\begin{align*}
			f(u_j^i)&=h(u_j)+k\eta(i-1) \text{ for }j=1,2,\ldots,k\eta,\\
			f(v_i)&=g(v_i)+nk\eta,
		\end{align*}
		for $i=1,2,\ldots,n$, where $S=\{1,2,\ldots,n(k\eta+1)\}$. Hence, $f$ is a bijective function.

		For the edges of $H^i$, for any $u_a^iu_b^i\in E(H^i)$, we have 
		$$f(u_a^i)+f(u_b^i)\equiv h(u_a)+h(u_b)\pmod{\eta},$$
		for some $a,b\in\{1,2,\ldots,k\eta\}$. Thus, by (\ref{eqcor3}) and (\ref{eqcor4}), $|\Gamma_q^H|$ edges of $H^i$ are labeled $q$, for $q=0,1$, under $f_\eta^*$, for $i=1,2,\ldots,n$.
		
		For the edges of $G$, for any $v_cv_d\in E(G)$, we have 
		$$f(v_c)+f(v_d)\equiv g(v_c)+g(v_d)\pmod{\eta},$$
		for some $c,d\in \{1,2,\ldots,n\}$. Applying (\ref{eqcor1}) and (\ref{eqcor2}), there are $|\Gamma_q^G|$ edges of $G$ with label $q$, for $q=0,1$, under $f_\eta^*$.
		
		For the edges of the form $v_iu_j^i$, we have
		\begin{equation}
			f(v_i)+f(u_j^i)\equiv g(v_i)+h(u_j)\pmod{\eta},	\label{eqcor6}
		\end{equation}
		for $i=1,2,\ldots,n$ and $j=1,2,\ldots,k\eta$. Define 
		$$\alpha_t^i=\bigcup_{u_a\in D_t}\{\xi:f(v_i)+f(u_a^i)\equiv \xi\pmod{\eta},\text{ }0\leq \xi<\eta\}$$
		for $t=0,1,\ldots,k-1$ and $i=1,2,\ldots,n$. Consequently, applying (\ref{eqcor6}),
		$$\alpha_t^i=\bigcup_{u_a\in D_t}\{\xi:g(v_i)+h(u_a)\equiv \xi\pmod{\eta},\text{ }0\leq \xi<\eta\}$$
		for $t=0,1,\ldots,k-1$ and $i=1,2,\ldots,n$.
		By (\ref{partition}) together with Remark~\ref{Dt}, we have 
		\begin{equation}
			\alpha_t^i=\{0,1,\ldots,\eta-1\},\label{alphacor}
		\end{equation}
		for $t=0,1,\ldots,k-1$ and $i=1,2,\ldots,n$. Note that Eq. (\ref{alphacor}) holds because, as $u_{a}$ ranges over $D_{t}$, the expression $g(v_i)+h(u_a)$ yields a permutation of the set $\{0,1,\ldots,\eta-1\}$ modulo $\eta$. Hence, there are $nk[\eta-\phi(\eta)]+nk|A_0|$ edges of the form $v_iu_j^i$ with label $0$ and $nk|A_1|$ edges with label $1$, under $f_\eta^*$.
		
		Therefore,
		\begin{align*}
			e_{f_\eta^*}(0)&=n|\Gamma_0^H|+|\Gamma_0^G|+nk[\eta-\phi(\eta)]+nk|A_0|,\\
			e_{f_\eta^*}(1)&=n|\Gamma_1^H|+|\Gamma_1^G|+nk|A_1|.
		\end{align*}
		Clearly, by Eq. (\ref{varpicor}),  $e_{f_\eta^*}(0)-e_{f_\eta^*}(1)=\varpi(G\circ H;g,h)$. Since the functions $g$ and $h$ satify (\ref{eqcor5}), it is immediate that $|e_{f_\eta^*}(0)-e_{f_\eta^*}(1)|\leq 1$. Hence, the conclusion follows.
	\end{proof}
	
	\begin{theorem}\label{thm2}
		Let $\eta\geq 3$ be an integer. Suppose that $G$ is a graph of order $n$ and size $m$. Also, let $S=\{1,2,\ldots,nk\eta\}$. In addition, let
		\begin{align}	
			\Gamma_0=&\{uv\in E(H):\gcd(h(u)+h(v),\eta)\neq 1\}\nonumber\\
			&\cup \{uv\in E(H):\gcd(h(u)+h(v),\eta)= 1\text{ and }\zeta_\eta(h(u)+h(v))=0\},\label{eqpro1}\\
			\Gamma_1=&\{uv\in E(H):\gcd(h(u)+h(v),\eta)=1\text{ and }\zeta_\eta(h(u)+h(v))=1\}.\label{eqpro2}
		\end{align}
		\begin{enumerate}
			\item[i.] Let
			\begin{equation}
				\varpi(G\otimes H;h)=mk^2\eta [\eta-\phi(\eta)+|A_0|-|A_1|]+n[|\Gamma_0|-|\Gamma_1|].\label{varpilex}
			\end{equation}
			If the function $h$ satisfies
			\begin{equation}
				\left|\varpi(G\otimes H;h)\right|\leq 1,\label{eqprolex}
			\end{equation}
			then the lexicographic product graph $G\otimes H$ is an arithmetic cordial graph modulo $\eta$ under $\langle S,\zeta_\eta,+\rangle$. Moreover, if $\varpi(G\otimes H;h)$ is equal to $0$, $1$, or $-1$, then $G\otimes H$ is a perfect, positive, or negative arithmetic cordial graph modulo $\eta$ under $\langle S,\zeta_\eta,+\rangle$, respectively.
			
			\item[ii.] Let $\eta$ be an odd integer and let
			\begin{equation}
				\varpi(G\square H;h)=mk [\eta-\phi(\eta)+|A_0|-|A_1|]+n[|\Gamma_0|-|\Gamma_1|].\label{varpicar}
			\end{equation} 
			If the function $h$ satisfies
			\begin{equation}
				\left|\varpi(G\square H;h)\right|\leq 1,\label{eqprocar}
			\end{equation}
			then the cartesian product graph $G\square H$ is an arithmetic cordial graph modulo $\eta$ under $\langle S,\zeta_\eta,+\rangle$. Moreover, if $\varpi(G\square H;h)$ is equal to $0$, $1$, or $-1$, then $G\square H$ is a perfect, positive, or negative arithmetic cordial graph modulo $\eta$ under $\langle S,\zeta_\eta,+\rangle$, respectively.
			
			\item[iii.] If $h$ is a perfect arithmetic cordial labeling modulo $\eta$ under $\langle S_1,\zeta_\eta,+\rangle$ of $H$, where $S_1=\{1,2,\ldots,k\eta\}$, then the tensor product graph $G\times H$ is a perfect arithmetic cordial graph modulo $\eta$ under $\langle S,\zeta_\eta,+\rangle$.
			
			\item[iv.] Let $\eta$ be an odd integer and let
			\begin{equation}
				\varpi(G\boxtimes H;h)=mk [\eta-\phi(\eta)+|A_0|-|A_1|]+(n+2m)[|\Gamma_0|-|\Gamma_1|].\label{varpistrong}
			\end{equation} 
			If the function $h$ satisfies
			\begin{equation}
				\left|\varpi(G\boxtimes H;h)\right|\leq 1,\label{eqprostrong}
			\end{equation}
			then the strong product graph $G\boxtimes H$ is an arithmetic cordial graph modulo $\eta$ under $\langle S,\zeta_\eta,+\rangle$. Moreover, if $\varpi(G\boxtimes H;h)$ is equal to $0$, $1$, or $-1$, then $G\boxtimes H$ is a perfect, positive, or negative arithmetic cordial graph modulo $\eta$ under $\langle S,\zeta_\eta,+\rangle$, respectively.
		\end{enumerate}
	\end{theorem}
	
	\begin{proof}
		Let $V(G)=\{v_1,v_2,\ldots,v_n\}$ and $V(H)=\{u_1,u_2,\ldots,u_{k\eta}\}$. Let $f:V(G\divideontimes H)\to S$ be a function defined by
		\begin{equation}
			f((v_i,u_j))=h(u_j)+k\eta (i-1)\label{biequationpro}
		\end{equation}
		for $j=1,2,\ldots,k\eta$ and $i=1,2,\ldots,n$, where $S=\{1,2,\ldots,nk\eta\}$ and $\divideontimes\in \{\otimes,\square,\times,\boxtimes\}$. Hence, $f$ is a bijective function.
		
		For (i), (ii), and (iv): for the edges of the form $(v_i,u_a)(v_i,u_b)$ where $u_au_b\in E(H)$ and $i=1,2,\ldots,n$, we have
		$$f((v_i,u_a))+f((v_i,u_b))\equiv h(u_a)+h(u_b)\pmod{\eta}.$$
		Clearly, by (\ref{eqpro1}) and (\ref{eqpro2}), $n|\Gamma_q|$ edges of the form $(v_i,u_a)(v_i,u_b)$ have label $q$ for $q=0,1$, under $f_\eta^*$.
		
		For the edges of the form $(v_c,u_a)(v_d,u_b)$, where $v_cv_d\in E(G)$ and the variables $a$ and $b$ depend on the product graphs defined in (i), (ii), (iii), and (iv), we have
		\begin{equation}
			f((v_c,u_a))+f((v_d,u_b))\equiv h(u_a)+h(u_b)\pmod{\eta}. \label{eqpro3}
		\end{equation}
		
		For (i), $a,b=1,2,\ldots,k\eta$. Define a set 
		$$\alpha_{a,t}^{v_cv_d}=\bigcup_{u_b\in D_t}\{\xi:f((v_c,u_a))+f((v_d,u_b))\equiv\xi\pmod{\eta},\text{ }0\leq\xi<\eta\}$$
		for each $v_cv_d\in E(G)$, $a=1,2,\ldots,k\eta$, and $t=0,1,\ldots,k-1$. Therefore, using Eq. (\ref{eqpro3}), we obtain
		$$\alpha_{a,t}^{v_cv_d}=\bigcup_{u_b\in D_t}\{\xi:h(u_a)+h(u_b)\equiv\xi\pmod{\eta},\text{ }0\leq\xi<\eta\}$$
		for each $v_cv_d\in E(G)$, $a=1,2,\ldots,k\eta$, and $t=0,1,\ldots,k-1$.
		By (\ref{partition}) together with Remark~\ref{Dt}, we have
		\begin{equation}
			\alpha_{a,t}^{v_cv_d}=\{0,1,\ldots,\eta-1\}\label{alphacar}
		\end{equation}
		for each $v_cv_d\in E(G)$, $a=1,2,\ldots,k\eta$, and $t=0,1,\ldots,k-1$. Observe that Eq. (\ref{alphacar}) holds since, as $u_{b}$ ranges over $D_{t}$, the expression $h(u_a)+h(u_b)$ results in a permutation of the set $\{0,1,\ldots,\eta-1\}$ modulo $\eta$. Hence,  $mk^2\eta[\eta-\phi(\eta)]+mk^2\eta|A_0|$ edges of the form $(v_c,u_a)(v_d,u_b)$ have label $0$ and  $mk^2\eta|A_1|$ edges have label $1$, under $f_\eta^*$.
		
		Therefore,
		\begin{align*}
			e_{f_\eta^*}(0)&=n|\Gamma_0|+mk^2\eta[\eta-\phi(\eta)]+mk^2\eta|A_0|,\\
			e_{f_\eta^*}(1)&=n|\Gamma_1|+mk^2\eta|A_1|.
		\end{align*}
		By Eq. (\ref{varpilex}), $e_{f_\eta^*}(0)-e_{f_\eta^*}(1)=\varpi(G\otimes H;h)$. Since $h$ satisfies (\ref{eqprolex}),  $|e_{f_\eta^*}(0)-e_{f_\eta^*}(1)|\leq 1$. So, (i) is true.
		
		For (ii), we have $a=b=1,2,\ldots,k\eta$. Thus, congruence (\ref{eqpro3}) becomes 
		\begin{equation}
			f((v_c,u_a))+f((v_d,u_b))\equiv 2h(u_a)\pmod{\eta}\label{eqpro4}
		\end{equation}
		for each $v_cv_d\in E(G)$ and $a=1,2,\ldots,k\eta$.
		Let
		$$\beta_{t}^{v_cv_d}=\bigcup_{u_a\in D_t}\{\xi:f((v_c,u_a))+f((v_d,u_b))\equiv\xi\pmod{\eta},\text{ }0\leq \xi<\eta\}$$
		for each $v_cv_d\in E(G)$ and $t=0,1,\ldots,k-1$. Using Eq. (\ref{eqpro4}), we have 
		$$\beta_{t}^{v_cv_d}=\bigcup_{u_a\in D_t}\{\xi:2h(u_a)\equiv\xi\pmod{\eta},\text{ }0\leq \xi<\eta\}$$
		for each $v_cv_d\in E(G)$ and $t=0,1,\ldots,k-1$.
		Note that $\eta$ is odd. By using (\ref{partition}) and Remark~\ref{Dt}, we obtain
		\begin{equation}
			\beta_{t}^{v_cv_d}=\{0,1,\ldots,\eta-1\}\label{beta}
		\end{equation}
		for each $v_cv_d\in E(G)$ and $t=0,1,\ldots,k-1$. Eq. (\ref{beta}) is true because, as $u_a$ ranges over $D_t$, the expression $2h(u_a)$ forms a permutation of the set $\{0,1,\ldots,\eta-1\}$ modulo $\eta$. This implies that $mk[\eta-\phi(\eta)]+mk|A_0|$ edges of the form $(v_c,u_a)(v_d,u_b)$ have label $0$ and  $mk|A_1|$ edges have label $1$, under $f_\eta^*$.
		
		Consequently,
		\begin{align}
			e_{f_\eta^*}(0)&=n|\Gamma_0|+mk[\eta-\phi(\eta)]+mk|A_0|,\label{carlabel0}\\
			e_{f_\eta^*}(1)&=n|\Gamma_1|+mk|A_1|.\label{carlabel1}
		\end{align}
		By Eq. (\ref{varpicar}), we have $e_{f_\eta^*}(0)-e_{f_\eta^*}(1)=\varpi(G\square H;h)$. Because $h$ satisfies (\ref{eqprocar}), $|e_{f_\eta^*}(0)-e_{f_\eta^*}(1)|\leq 1$. Hence, (ii) is true.
		
		For (iii), we have $u_au_b\in E(H)$. So, $(v_c,u_b)(v_d,u_a)\in E(G\times H)$ and as a result
		$$f((v_c,u_b))+f((v_d,u_a))\equiv h(u_a)+h(u_b)\pmod{\eta}.$$
		Combining this with congruence (\ref{eqpro3}) and applying (\ref{eqpro1}) and (\ref{eqpro2}), we have
		\begin{equation}
			e_{f_\eta^*}(q)=2m|\Gamma_q|\label{eqproten}
		\end{equation}
		for $q=0,1$. 
		
		Now, suppose that $h$ is a perfect arithmetic cordial labeling modulo $\eta$ under $\langle S_1,\zeta_\eta,+\rangle$ of $H$, where $S_1=\{1,2,\ldots,k\eta\}$. Clearly, by (\ref{eqpro1}) and (\ref{eqpro2}), 
		$$e_{h_\eta^*}(q)=|\Gamma_q|,$$
		for $q=0,1$, and so $e_{h_\eta^*}(0)-e_{h_\eta^*}(1)=|\Gamma_0|-|\Gamma_1|=0$.
		
		Hence, by (\ref{eqproten}), it follows that $e_{f_\eta^*}(0)-e_{f_\eta^*}(1)=0$. Therefore, (iii) is true.
		
		For (iv), because $E(G\boxtimes H)=E(G\square H)\cup E(G\times H)$ and $E(G\square H)\cap E(G\times H)=\varnothing$, we can combine the arguments in (ii) (Eq. (\ref{carlabel0}) and (\ref{carlabel1})) and (iii) (Eq. (\ref{eqproten})), and so
		\begin{align*}
			e_{f_\eta^*}(0)&=n|\Gamma_0|+mk[\eta-\phi(\eta)]+mk|A_0|+2m|\Gamma_0|,\\
			e_{f_\eta^*}(1)&=n|\Gamma_1|+mk|A_1|+2m|\Gamma_1|.
		\end{align*}
		By Eq. (\ref{varpistrong}), we have $e_{f_\eta^*}(0)-e_{f_\eta^*}(1)=\varpi(G\boxtimes H;h)$. Applying the inequality in (\ref{eqprostrong}), which is satisfied by $h$, we have $|e_{f_\eta^*}(0)-e_{f_\eta^*}(1)|\leq 1$. Therefore, (iv) is true.
	\end{proof}
	
	The result for the tensor product graph stated in Theorem~\ref{thm2} (iii) can easily be extended from $+$ to an arbitrary commutative binary function $\star$ with the property that if $a\equiv b\pmod{\eta}$ and $c\equiv d\pmod{\eta}$, then $a\star c\equiv b\star d\pmod{\eta}$, for a fixed positive integer $\eta$. Hence, we have the following corollary.

	\begin{corollary}\label{corten}
		Let $\eta\geq 3$ be an integer and suppose that $G$ is a graph of order $n$, $n\geq 2$. Let $\star$ be a commutative binary function with the property that if $a\equiv b\pmod{\eta}$ and $c\equiv d\pmod{\eta}$, then $a\star c\equiv b\star d\pmod{\eta}$. If $H$ is a perfect arithmetic cordial graph modulo $\eta$ under $\langle S_1,\zeta_\eta,\star\rangle$ with order $k\eta$, where $S_1=\{1,2,\ldots,k\eta\}$ and $k\geq 1$, then $G\times H$ is a perfect arithmetic cordial graph modulo $\eta$ under $\langle S,\zeta_\eta,\star\rangle$, where $S=\{1,2,\ldots,nk\eta\}$.
	\end{corollary}
	
	\begin{proof}
		Let $h$ be a perfect arithmetic cordial labeling modulo $\eta$ under $\langle S_1,\zeta_\eta,\star\rangle$ of $H$, where $S_1=\{1,2,\ldots,k\eta\}$. Consider the assumptions given to $G$ and $H$ in the proof of Theorem~\ref{thm2} and define the bijective function $f:V(G\times H)\to S$ similarly to Eq. (\ref{biequationpro}), where $S=\{1,2,\ldots,nk\eta\}$. Consequently, the proof follows analogously to that of Theorem~\ref{thm2} (iii) by replacing the binary operation $+$ with the binary function $\star$, except for Eq. (\ref{biequationpro}).
	\end{proof}

	\begin{property}\label{pro2}
		Suppose that $\zeta_\eta$ satifies the property $|A_0|=|A_1|$. Hence, $|A_0|=|A_1|=\frac{\phi(\eta)}{2}$ and $|A_0|-|A_1|=0$.
	\end{property}

	In the following results, assume that $\zeta_\eta$ satisfy Property~\ref{pro2}. 
	
	\begin{corollary}\label{cor2}
		Suppose that $p$ is an odd prime with $p\geq 3$. Let $H$ be a negative arithmetic cordial graph modulo $p$ under $\langle S_1,\zeta_p, +\rangle$ of order $p$, where $S_1=\{1,2,\ldots,p\}$. If $G$ is a perfect (positive) (negative) arithmetic cordial graph modulo $p$ under $\langle S_2,\zeta_p, +\rangle$ of order $n$, where $S_2=\{1,2,\ldots,n\}$, then $G\circ H$ is a perfect (positive) (negative) arithmetic cordial graph modulo $p$ under $\langle S_3,\zeta_p,+\rangle$, where $S_3=\{1,2,\ldots,n(p+1)\}$. 
	\end{corollary}
	
	\begin{proof}
		Suppose that $h:V(H)\to S_1$ is a negative arithmetic cordial graph modulo $p$ under $\langle S_1,\zeta_p, +\rangle$, where $S_1=\{1,2,\ldots,p\}$. In addition, let $g:V(G)\to S_2$ be an arithmetic cordial labeling modulo $p$ under $\langle S_2,\zeta_p, +\rangle$ of $G$, where $S_2=\{1,2,\ldots,n\}$. Using the sets $\Gamma_q^G$ and $\Gamma_q^H$, for $q=0,1$, defined in Theorem~\ref{thm1}, we have $e_{g_p^*}(q)=|\Gamma_q^G|$ and $e_{h_p^*}(q)=|\Gamma_q^H|$. Consequently,
		\begin{align}
			|\Gamma_0^G|-|\Gamma_1^G|&=e_{g_p^*}(0)-e_{g_p^*}(1)=\begin{cases}
				-1&\text{ if $g$ is negative arithmetic cordial labeling}\\
				0&\text{ if $g$ is perfect arithmetic cordial labeling}\\
				1&\text{ if $g$ is positive arithmetic cordial labeling},
			\end{cases}\label{proofGcircH1}\\
			|\Gamma_0^H|-|\Gamma_1^H|&=e_{h_p^*}(0)-e_{h_p^*}(1)=-1.\label{proofGcircH2}
		\end{align}
		Because $p$ is an odd prime, we have $p-\phi(p)=p-(p-1)=1$. 
		
		Using Eq. (\ref{varpicor}) with $k=1$ and $\eta=p$, Eq. (\ref{proofGcircH2}) implies that 
		\begin{align*}
			\varpi(G\circ H;g,h)=n[p-\phi(p)]+n[|\Gamma_0^H|-|\Gamma_1^H|]+|\Gamma_0^G|-|\Gamma_1^G|=n-n+|\Gamma_0^G|-|\Gamma_1^G|=|\Gamma_0^G|-|\Gamma_1^G|.
		\end{align*} 
		By Eq. (\ref{proofGcircH1}),
		$$\varpi(G\circ H;g,h)=\begin{cases}
			-1&\text{ if $g$ is negative arithmetic cordial labeling}\\
			0&\text{ if $g$ is perfect arithmetic cordial labeling}\\
			1&\text{ if $g$ is positive arithmetic cordial labeling}.
		\end{cases}$$
		Thus, by Theorem~\ref{thm1}, the conclusion is true. 
	\end{proof}

	\begin{corollary}\label{cor3}
		Let $\eta\geq 3$ and let $G$ be a graph of order $n$ and size $m$. Suppose that $H$ is a negative arithmetic cordial graph modulo $\eta$ under $\langle S_1,\zeta_\eta,+\rangle$ of order $k\eta$, where $S_1=\{1,2,\ldots,k\eta\}$. Let $S=\{1,2,\ldots,nk\eta\}$. 
		
		\begin{enumerate}
			\item[i.] For $n,m,k\geq 1$, if
			\begin{equation}
				\frac{n}{m}=k^2\eta[\eta-\phi(\eta)], \label{corpro1}
			\end{equation}
			then the lexicographic product graph $G\otimes H$ is a perfect arithmetic cordial graph modulo $\eta$ under $\langle S,\zeta_\eta,+\rangle$.
			
			\item[ii.] Let $\eta$ be an odd integer. For $n,m,k\geq 1$, if
			\begin{equation}
				\frac{n}{m}=k[\eta-\phi(\eta)],\label{corpro2}
			\end{equation}
			then the cartesian product graph $G\square H$ is a perfect arithmetic cordial graph modulo $\eta$ under $\langle S,\zeta_\eta,+\rangle$.
			
			\item[iii.] Let $\eta$ be an odd integer. For $n,m\geq 1$ and $k\geq 3$, if
			\begin{equation}
				\frac{n}{m}=k[\eta-\phi(\eta)]-2,\label{corpro3}
			\end{equation}
			then the strong product graph $G\boxtimes H$ is a perfect arithmetic cordial graph modulo $\eta$ under $\langle S,\zeta_\eta,+\rangle$.
		\end{enumerate}

	\end{corollary}

	\begin{proof} 
		Let $h:V(H)\to S_1$ be a negative arithmetic cordial labeling modulo $\eta$ under $\langle S_1,\zeta_\eta,+\rangle$ of $H$, where $S_1=\{1,2,\ldots,k\eta\}$. We will use the sets $\Gamma_0$ and $\Gamma_1$ defined in Theorem~\ref{thm2}. It is clear that $e_{h_\eta^*}(q)=|\Gamma_q|$ for $q=0,1$. Hence, 
		\begin{equation}
			e_{h_\eta^*}(0)-e_{h_\eta^*}(1)=|\Gamma_0|-|\Gamma_1|=-1.\label{corpro4}
		\end{equation}
		
		For (i), suppose that Eq. (\ref{corpro1}) holds, and so $n=mk^2\eta[\eta-\phi(\eta)]$. Substitute this to Eq. (\ref{varpilex}) together with Eq. (\ref{corpro4}), we obtain
		$$\varpi(G\otimes H;h)=mk^2\eta[\eta-\phi(\eta)]+n[|\Gamma_0|-|\Gamma_1|]=mk^2\eta[\eta-\phi(\eta)]-mk^2\eta[\eta-\phi(\eta)]=0.$$
		By Theorem~\ref{thm2} (i), the conclusion holds.
		
		For (ii), assume that Eq. (\ref{corpro2}) holds. So, $n=mk[\eta-\phi(\eta)]$. Along with Eq. (\ref{corpro4}), Eq. (\ref{varpicar}) becomes
		$$\varpi(G\square H;h)=mk [\eta-\phi(\eta)]+mk[\eta-\phi(\eta)][|\Gamma_0|-|\Gamma_1|]=mk [\eta-\phi(\eta)]-mk [\eta-\phi(\eta)]=0.$$
		Applying Theorem~\ref{thm2} (ii), the conclusion follows.
		
		Lastly, for (iii), suppose that Eq. (\ref{corpro3}) holds. It follows that $n=mk[\eta-\phi(\eta)]-2m$. Using Eq. (\ref{varpistrong}) together with  Eq. (\ref{corpro4}), we have
		$$\varpi(G\boxtimes H;h)=mk [\eta-\phi(\eta)]+(mk[\eta-\phi(\eta)]-2m+2m)[|\Gamma_0|-|\Gamma_1|]=mk [\eta-\phi(\eta)]-mk[\eta-\phi(\eta)]=0.$$
		Using Theorem~\ref{thm2} (iv), the conclusion holds.
	\end{proof}

	\section{Applications to Some Graphs}
	
	Suppose that $p$ is an odd prime and let $\zeta_p$ satisfies Properties~\ref{pro1} and \ref{pro2}. It should be noted that $p-\phi(p)=p-(p-1)=1$. 
	
	We extend some results in \cite{Andoyo1} to the context of arithmetic cordial labeling using the same construction from their respective proof. 
	
	\begin{lemma}[General version of Lemma 3.1 in \cite{Andoyo1}]\label{lem1}
		There exists an integer $k$, $0\leq k<p$, for which $\zeta_p\left(\frac{p+2k\pm1}{2}\right)=1$.
	\end{lemma}
	
	\begin{proof}
		For each $k=0,1,\ldots,p-1$, define $\frac{p+2k+1}{2}\equiv \theta_k\pmod{p}$, $0\leq \theta_k<p$. Suppose that $$\Lambda = \{\theta_k:k=0,1,\ldots,p-1\}.$$
		We have
		$$\Lambda=\left\{\frac{p+1}{2},\frac{p+3}{2},\ldots,p-1,0,1,\ldots,\frac{p-1}{2}\right\}.$$
		Hence, $\Lambda=\{0\}\cup \{1,2,\ldots,p-1\}$. Note that by Property~\ref{pro1}, $\zeta_p\left(\frac{p+2k+1}{2}\right)=\zeta_p(\theta_k)$ for each $k=0,1,\ldots,p-1$, $k\neq \frac{p-1}{2}$. Hence, Property~\ref{pro2} guarantees that for some $k$, $\theta_k\in A_1$, that is, $\zeta_p(\theta_k)=1$, $0\leq k<p$. Hence, the conclusion follows.
		
		The proof for $\frac{p+2k-1}{2}$ is analogous.
	\end{proof}
	
	\begin{remark} \label{rem1}
		Let $p\geq 3$. The following statements are true:
		\begin{enumerate}
			\item[i.] The path graph $P_p$ is a perfect arithmetic cordial graph modulo $p$ under $\langle S,\zeta_p,+\rangle$, where $S=\{1,2,\ldots,p\}$.
			
			\item[ii.] The cycle graph $C_p$ is a positive arithmetic cordial graph modulo $p$ under $\langle S,\zeta_p,+\rangle$, where $S=\{1,2,\ldots,p\}$.

			\item[iii.] The star graph $\text{Star}_p$ is a perfect arithmetic cordial graph modulo $p$ under $\langle S,\zeta_p,+\rangle$, where $S=\{1,2,\ldots,p\}$.
			
			\item[iv.] The tadpole graph $T_{p,p}$ is a perfect arithmetic cordial graph modulo $p$ under $\langle S,\zeta_p,+\rangle$, where $S=\{1,2,\ldots,2p\}$.

			\item[v.] The kayak paddle graph $KP_{p,p,p}$ is a perfect arithmetic cordial graph modulo $p$ under $\langle S,\zeta_p,+\rangle$, where $S=\{1,2,\ldots,3p\}$.

			\item[vi.] The bistar graph $B_{p,p}$ is a positive arithmetic cordial graph modulo $p$ under $\langle S,\zeta_p,+\rangle$, where $S=\{1,2,\ldots,2p\}$.

			\item[vii.] The fan graph graph $F_{p+1}$ is a positive arithmetic cordial graph modulo $p$ under $\langle S,\zeta_p,+\rangle$, where $S=\{1,2,\ldots,p+1\}$.
		\end{enumerate}
	\end{remark}
	
	In star graph, we use the notation $\text{Star}_p$ instead of $S_p$ to avoid any confusion with the sets of the form $S_j$ that is defined for arithmetic structures. For the definition of star, kayak paddle, bistar, and fan graphs, refer to \cite{Andoyo1}. Note that Lemma~\ref{lem1} is used for the proof of Remark~\ref{rem1} (iv) and (v).
	
	Moreover, the following result is obtained from \cite{Andoyo5}.
	
	\begin{remark}\label{rem2} 
		Let $p\geq 3$. The following statements are true:
		\begin{enumerate}
			\item[i.] The ladder graph $L_p$ is a positive arithmetic cordial graph modulo $p$ under $\langle S,\zeta_p,+\rangle$, where $S=\{1,2,\ldots,2p\}$.
			\item[ii.] Let $n\geq 2$. Then the alternate cycle snake graph $A_n(C_p)$ is a positive arithmetic cordial graph modulo $p$ under $\langle S,\zeta_p,+\rangle$, where $S=\{1,2,\ldots,np\}$.
			\item[iii.] Suppose that $\zeta_p(1)=1$ and $\zeta_p(2)=0$. Let $J$ be a connected graph of order $k$, $k\geq 2$. If $J$ has size $k+1$, then $J\circ P_{p-1}$ is a negative arithmetic cordial graph modulo $p$ under $\langle S,\zeta_p,+\rangle$, where $S=\{1,2,\ldots,kp\}$. 
		\end{enumerate}
	\end{remark}

	For the definition of ladder and alternate cycle snake graphs, refer to \cite{Andoyo5}.
	
	\begin{theorem}\label{T_p-1,1}
		Let $p\geq 5$ with $\zeta_p(1)=1$. The tadpole graph $T_{p-1,1}$ is a negative arithmetic cordial graph modulo $p$ under $\langle S,\zeta_p,+\rangle$, where $S=\{1,2,\ldots,p\}$.
	\end{theorem}
	
	\begin{proof}
		Suppose that $C_{p-1}$ and $P_2$ are cycle and path of $T_{p-1,1}$, respectively, with $V(C_{p-1})=\{v_1,v_2,\ldots,v_{p-1}\}$ and $V(P_2)=\{v_{(p+1)/2},u_1\}$. Define a function $f:V(T_{p-1,1})\to S$ by
		\begin{align*}
			f(v_i)&=\begin{cases}
				i&\text{ for }i=1,2,\ldots,\frac{p-1}{2}\\
				i+1&\text{ for }i=\frac{p+1}{2},\frac{p+3}{2},\ldots,p-1,
			\end{cases}\\
			f(u_1)&=\frac{p+1}{2},
		\end{align*}
		where $S=\{1,2,\ldots,p\}$. Hence, $f$ is a bijective function.
		
		For the edges,
		\begin{align*}
			f(v_i)+f(v_{i+1})&\equiv \begin{cases}
				2i+1\pmod{p}&\text{ for }i=1,2,\ldots,\frac{p-3}{2}\\
				1\pmod{p}&\text{ for }i=\frac{p-1}{2}\\
				2i+3-p\pmod{p}&\text{ for }i=\frac{p+1}{2},\frac{p+3}{2}\ldots,p-2,
			\end{cases}\\
			f(v_1)+f(v_{p-1})&\equiv 1\pmod{p},\\
			f(v_{(p+1)/2})+f(u_1)&\equiv 2\pmod{p}.
		\end{align*}
		
		For the edges in $\{v_{(p+1)/2}u_1\}\cup [E(C_{p-1})-\{v_1v_{p-1}\}]$, define 
		\begin{align*}
			\alpha&=\bigcup_{i=1}^{p-2}\{\xi:f(v_i)+f(v_{i+1})\equiv \xi\pmod{p},\text{ }0<\xi<p\},\\
			\beta&=\{\xi:f(v_{(p+1)/2})+f(u_1)\equiv \xi\pmod{p},\text{ }0<\xi<p\}.
		\end{align*}
		Thus,
		\begin{align*}
			\alpha&=\left[\bigcup_{i=1}^{(p-3)/2}\{2i+1\}\right]\cup\{1\}\cup \left[\bigcup_{i=(p+1)/2}^{p-2}\{2i+3-p\}\right]\\
			&=\{1,3,\ldots,p-2\}\cup\{4,6,\ldots,p-1\}\\
			&=\{1\}\cup\{3,4,\ldots,p-1\}
		\end{align*}
		and $\beta=\{2\}$. Hence, $\alpha\cup \beta=\{1,2,\ldots,p-1\}$. For the edge $v_1v_{p-1},$ since $\zeta_p(1)=1$, we have $f_p^*(v_1v_{p-1})=1$. Using Property~\ref{pro2} together with the fact that $f_p^*(v_1v_{p-1})=1$, 
		$$e_{f_p^*}(0)=\frac{p-1}{2}\text{ and }e_{f_p^*}(1)=\frac{p-1}{2}+1,$$
		and so $e_{f_p^*}(0)-e_{f_p^*}(1)=-1$. Therefore, $T_{p-1,1}$ is a negative arithmetic cordial graph modulo $p$ under $\langle S,\zeta_p,+\rangle$.
	\end{proof}
	
	\begin{proposition}\label{prop1}
		Let $p\geq 5$ and let $\zeta_p(1)=1$. Suppose that $G$ is a graph and let $S=\{1,2,\ldots,(p+1)|V(G)|\}$. The following statements are true:
		\begin{enumerate}
			\item[i.] If $G\in \{P_p,\text{Star}_p,T_{p,p},KP_{p,p,p}\}$, then the corona graph $G\circ T_{p-1,1}$ is a perfect arithmetic cordial graph modulo $p$ under $\langle S,\zeta_p,+\rangle$.
			
			\item[ii.] Let $n\geq 2$ for the graph $A_n(C_p)$. If $G\in \{C_p,B_{p,p},F_{p+1},L_p,A_n(C_p)\}$, then the corona graph $G\circ T_{p-1,1}$ is a positive arithmetic cordial graph modulo $p$ under $\langle S,\zeta_p,+\rangle$.
		\end{enumerate}
	\end{proposition}
	
	\begin{proof}
		Suppose that $S=\{1,2,\ldots,(p+1)|V(G)|\}$. Note that $T_{p-1,1}$ is a negative arithmetic cordial graph modulo $p$ under $\langle S_1,\zeta_p,+\rangle$ of order $p$, where $S_1=\{1,2,\ldots,p\}$, according to Theorem~\ref{T_p-1,1}.
		
		For (i), notice that Remark~\ref{rem1} (i), (iii), (iv), and (v) show that $P_p$, $\text{Star}_p$, $T_{p,p}$, and $KP_{p,p,p}$, respectively, are perfect arithmetic cordial graphs modulo $p$ under their respective arithmetic structures. Hence, for any $G\in \{P_p,\text{Star}_p,T_{p,p},KP_{p,p,p}\}$, $G$ is a perfect arithmetic cordial graph modulo $p$ under $\langle S_2,\zeta_p,+\rangle$, where $S_2=\{1,2,\ldots,|V(G)|\}$. Using Corollary~\ref{cor2} with $G\in \{P_p,\text{Star}_p,T_{p,p},KP_{p,p,p}\}$ and $H=T_{p-1,1}$, we conclude that $G\circ T_{p-1,1}$ is a perfect arithmetic cordial graph modulo $p$ under $\langle S,\zeta_p,+\rangle$.
		
		For (ii), by Remark~\ref{rem1} (ii), (vi), and (vii), the graphs $C_p$, $B_{p,p}$, and $F_{p+1}$, respectively, are positive arithmetic cordial graphs modulo $p$ under their respective arithmetic structures. Also, by Remark~\ref{rem2} (i) and (ii), the graphs $L_p$ and $A_n(C_p)$, respectively, are positive arithmetic cordial graphs modulo $p$ under their respective arithmetic structures. This implies that for any $G\in \{C_p,B_{p,p},F_{p+1},L_p,A_n(C_p)\}$, $G$ is a positive arithmetic cordial graph modulo $p$ under $\langle S_3,\zeta_p,+\rangle$, where $S_3=\{1,2,\ldots,|V(G)|\}$. Applying Corollary~\ref{cor2} with $G\in \{C_p,B_{p,p},F_{p+1},L_p,A_n(C_p)\}$ and $H=T_{p-1,1}$, it follows that $G\circ T_{p-1,1}$ is a positive arithmetic cordial graph modulo $p$ under $\langle S,\zeta_p,+\rangle$.
	\end{proof}
	
	\begin{proposition}\label{prop2}
		Let $p\geq 5$ and let $\zeta_p(1)=1$. Then the corona graph $T_{p-1,1}\circ T_{p-1,1}$ is a negative arithmetic cordial graph modulo $p$ under $\langle S,\zeta_p,+\rangle$, where $S=\{1,2,\ldots,p(p+1)\}$.
	\end{proposition}
	
	\begin{proof}
		According to Theorem~\ref{T_p-1,1}, $T_{p-1,1}$ is a negative arithmetic cordial graph modulo $p$ under $\langle S_1,\zeta_p,+\rangle$, where $S_1=\{1,2,\ldots,p\}$. Because $T_{p-1,1}$ has order $p$, using Corollary~\ref{cor2} with $G=H=T_{p-1,1}$, it is clear that $T_{p-1,1}\circ T_{p-1,1}$ is a negative arithmetic cordial graph modulo $p$ under $\langle S,\zeta_p,+\rangle$, where $S=\{1,2,\ldots,p(p+1)\}$.
	\end{proof}
	
	\begin{proposition}\label{prop3}
		Let $p\geq 5$ and suppose that $\zeta_p(1)=1$ and $\zeta_p(2)=0$. Let $J$ be a connected graph of order $k$ and size $k+1$, $k\geq 2$. Then the corona graph $(J\circ P_{p-1})\circ T_{p-1,1}$ is a negative arithmetic cordial graph modulo $p$ under $\langle S,\zeta_p,+\rangle$, where $S=\{1,2,\ldots,kp(p+1)\}$.
	\end{proposition}
	
	\begin{proof}
		Remark~\ref{rem2} (iii) states that $J\circ P_{p-1}$ is a negative arithmetic cordial graph modulo $p$ under $\langle S_1,\zeta_p,+\rangle$, where $S_1=\{1,2,\ldots,kp\}$. In addition, Theorem~\ref{T_p-1,1} states that $T_{p-1,1}$ is a negative arithmetic cordial graph modulo $p$ under $\langle S_2,\zeta_p,+\rangle$, where $S_2=\{1,2,\ldots,p\}$. Because $T_{p-1,1}$ has order $p$, using Corollary~\ref{cor2} with $G=J\circ P_{p-1}$ and $H=T_{p-1,1}$, it is evident that $(J\circ P_{p-1})\circ T_{p-1,1}$ is a negative arithmetic cordial graph modulo $p$ under $\langle S,\zeta_p,+\rangle$, where $S=\{1,2,\ldots,kp(p+1)\}$.
	\end{proof}

	\begin{proposition}\label{prop4}
		Suppose that $p\geq 3$ and let $G$ be a graph of order $n$, $n\geq 2$. If $H\in \{P_p,\text{Star}_p,T_{p,p},KP_{p,p,p}\}$, then the tensor product graph $G\times H$ is a perfect arithmetic cordial graph modulo $p$ under $\langle S_1,\zeta_p,+\rangle$, where $S_1=\{1,2,\ldots,n|V(H)|\}$.
	\end{proposition}
	
	\begin{proof}
		According to Remark~\ref{rem1} (i), (iii), (iv), and (v), the graphs $P_p$, $\text{Star}_p$, $T_{p,p}$, and $KP_{p,p,p}$, respectively, are perfect arithmetic cordial graphs modulo $p$ under their respective arithmetic structures. Hence, for each $H\in \{P_p,\text{Star}_p,T_{p,p},KP_{p,p,p}\}$, $H$ is a perfect arithmetic cordial graph modulo $p$ under $\langle S_4,\zeta_p,+\rangle$, where $S_4=\{1,2,\ldots,|V(H)|\}$. Observe that $|V(P_p)|=|V(\text{Star}_p)|=p$, $|V(T_{p,p})|=2p$, and $|V(KP_{p,p,p})|=3p$. Thus, the order of $H$ is $kp$ for $k\in\{1,2,3\}$. Using Corollary~\ref{corten} with $\star$ defined to be the binary operation of addition, $H\in \{P_p,\text{Star}_p,T_{p,p},KP_{p,p,p}\}$, $k\in \{1,2,3\}$, and $\eta=p$, we can conclude that $G\times H$ is a perfect arithmetic cordial graph modulo $p$ under $\langle S_1,\zeta_p,+\rangle$, where $S_1=\{1,2,\ldots,n|V(H)|\}$.
	\end{proof}
	
	\begin{proposition}\label{prop5}
		Let $p\geq 5$ with $\zeta_p(1)=1$. Suppose that $G$ is a graph with size $m$, $m\geq 1$. The following statements are true:
		\begin{enumerate}
			\item[i.] If $G$ has order $mp$, then the lexicographic product graph $G\otimes T_{p-1,1}$ is a perfect arithmetic cordial graph modulo $p$ under $\langle S_1,\zeta_p,+\rangle$, where $S_1=\{1,2,\ldots,mp^2\}$.
			\item[ii.] If $G$ has order $m$, then the cartesian product graph $G\square T_{p-1,1}$ is a perfect arithmetic cordial graph modulo $p$ under $\langle S_2,\zeta_p,+\rangle$, where $S_2=\{1,2,\ldots,mp\}$.
		\end{enumerate}
	\end{proposition}

	\begin{proof}
		Note that $T_{p-1,1}$ is a negative arithmetic cordial graph modulo $p$ under $\langle S_3,\zeta_p,+\rangle$ of order $p$, where $S_3=\{1,2,\ldots,p\}$ in accordance with Theorem~\ref{T_p-1,1} and we apply Corollary~\ref{cor3} with $H=T_{p-1,1}$. Suppose that $G$ has order $n$. Also, note that $p-\phi(p)=1$
		
		For (i), let $n=mp$. Using Eq. (\ref{corpro1}) with $n=mp$, $k=1$, and $\eta=p$, we have
		\begin{align*}
			\frac{n}{m}&=k^2\eta[\eta-\phi(\eta)]\\
			\frac{mp}{m}&=1^2p[p-\phi(p)]\\
			p&=p.
		\end{align*}
		Hence, Eq. (\ref{corpro1}) holds and by Corollary~\ref{cor3} (i), $G\otimes T_{p-1,1}$ is a perfect arithmetic cordial graph modulo $p$ under $\langle S_1,\zeta_p,+\rangle$, where $S_1=\{1,2,\ldots,mp^2\}$.
		
		For (ii), let $n=m$. Applying Eq. (\ref{corpro2}) with $n=m$, $k=1$, and $\eta=p$, we have
		\begin{align*}
			\frac{n}{m}&=k[\eta-\phi(\eta)]\\
			\frac{m}{m}&=[p-\phi(p)]\\
			1&=1.
		\end{align*}
		So, Eq. (\ref{corpro2}) is satisfied. Therefore, by Corollary~\ref{cor3} (ii), $G\square T_{p-1,1}$ is a perfect arithmetic cordial graph modulo $p$ under $\langle S_2,\zeta_p,+\rangle$, where $S_2=\{1,2,\ldots,mp\}$.
	\end{proof}
	
	Let $p\geq 5$ and set $\zeta_p(1)=1$. Suppose that $m\geq 4$. According to Proposition~\ref{prop5} (ii), the cartesian product graphs $C_m\square T_{p-1,1}$ and $T_{q,r}\square T_{p-1,1}$ with $q+r=m$, $q\geq 3$ and $r\geq 1$, are arithmetic cordial graphs modulo $p$ under $\langle S,\zeta_p,+\rangle$, where $S=\{1,2,\ldots,mp\}$.

	\begin{proposition}\label{prop6}
		Let $G$ be a graph of size $m$, $m\geq 1$, and let $J$ be a connected graph of order $k$ and size $k+1$. Suppose further that $p\geq 3$ with $\zeta_p(1)=1$ and $\zeta_p(2)=0$. The following statements are true:
		\begin{enumerate}
			\item[i.] Let $k\geq 2$. If $G$ has order $mpk^2$, then the lexicographic product graph $G\otimes (J\circ P_{p-1})$ is a perfect arithmetic cordial graph modulo $p$ under $\langle S_1,\zeta_p,+\rangle$, where $S_1=\{1,2,\ldots,mp^2k^3\}$.
			\item[ii.] Let $k\geq 2$. If $G$ has order $mk$, then the cartesian product graph $G\square (J\circ P_{p-1})$ is a perfect arithmetic cordial graph modulo $p$ under $\langle S_2,\zeta_p,+\rangle$, where $S_2=\{1,2,\ldots,mpk^2\}$.
			\item[iii.] Let $k\geq 3$. If $G$ has order $m(k-2)$, then the strong product graph $G\boxtimes (J\circ P_{p-1})$ is a perfect arithmetic cordial graph modulo $p$ under $\langle S_3,\zeta_p,+\rangle$, where $S_3=\{1,2,\ldots,mkp(k-2)\}$. 
		\end{enumerate}
	\end{proposition}
	
	\begin{proof}
		Let $n$ be the order of $G$. According to Remark~\ref{rem2} (iii), $J\circ P_{p-1}$ is a negative arithmetic cordial graph modulo $p$ under $\langle S_4,\zeta_p,+\rangle$ of order $kp$, where $S_4=\{1,2,\ldots,kp\}$ and we apply Corollary~\ref{cor3} with $H=J\circ P_{p-1}$. Note that $p-\phi(p)=1$.
		
		For (i), let $n=mpk^2$. Using Eq. (\ref{corpro1}) with $n=mpk^2$ and $\eta=p$, observe the following calculation:
		\begin{align*}
			\frac{n}{m}&=k^2\eta[\eta-\phi(\eta)]\\
			\frac{mpk^2}{m}&=k^2p[p-\phi(p)]\\
			pk^2&=k^2p.
		\end{align*}
		Thus, Eq. (\ref{corpro1}) holds. According to Corollary~\ref{cor3} (i), $G\otimes (J\circ P_{p-1})$ is a perfect arithmetic cordial graph modulo $p$ under $\langle S_1,\zeta_p,+\rangle$, where $S_1=\{1,2,\ldots,mp^2k^3\}$.
		
		For (ii), let $n=mk$. By applying Eq. (\ref{corpro2}) with $n=mk$ and $\eta=p$, we have
		\begin{align*}
			\frac{n}{m}&=k[\eta-\phi(\eta)]\\
			\frac{mk}{m}&=k[p-\phi(p)]\\
			k&=k.
		\end{align*}
		Hence, Eq. (\ref{corpro2}) holds and it follows that $G\square (J\circ P_{p-1})$ is a perfect arithmetic cordial graph modulo $p$ under $\langle S_2,\zeta_p,+\rangle$, where $S_2=\{1,2,\ldots,mpk^2\}$ by Corollary~\ref{cor3} (ii). 
		
		Lastly, for (iii), let $n=m(k-2)$. Take $n=m(k-2)$ and $\eta=p$ and substitute to Eq. (\ref{corpro3}). So,
		\begin{align*}
			\frac{n}{m}&=k[\eta-\phi(\eta)]-2\\
			\frac{m(k-2)}{m}&=k[p-\phi(p)]-2\\
			k-2&=k-2.
		\end{align*}
		Hence, Eq. (\ref{corpro3}) is satisfied. This means that $G\boxtimes (J\circ P_{p-1})$ is a perfect arithmetic cordial graph modulo $p$ under $\langle S_3,\zeta_p,+\rangle$, where $S_3=\{1,2,\ldots,mkp(k-2)\}$, in accordance with Corollary~\ref{cor3} (iii).
	\end{proof}
	
	The preceding propositions serve as direct and easy applications of Legendre cordial labeling \cite{Andoyo1} and logarithmic cordial labeling \cite{Andoyo3} to certain product graphs. For Legendre cordial labeling, suppose that $\zeta_p^1(a)=\frac{1+(a/p)}{2}$, where $(a/p)$ is the Legendre symbol over $p$ (see Definition 2.8 in \cite{Andoyo1}). As shown in \cite{Andoyo5}, $\zeta_p^1$ satisfies Properties~\ref{pro1} and \ref{pro2}. Clearly, $\zeta_p^1(1)=1$. In the case where $\zeta_p^1(2)=0$, assume that $p\equiv \pm 3\pmod{8}$ (see Theorem 2.4 in \cite{Andoyo3}). Hence, the graphs described in Propositions~\ref{prop1} through \ref{prop3}, Proposition~\ref{prop4} (i), and Propositions~\ref{prop5} through \ref{prop6} are all Legendre cordial graphs modulo $p$. Now, for logarithmic cordial labeling, let $\zeta_p^2(a)\equiv \ind_{\pi,p}(a)\pmod{2}$, where $\pi$ is a primitive root of $p$ (see Definition 2.4 in \cite{Andoyo3}). Thus, $\zeta_p^2$ satisfies Properties~\ref{pro1} and \ref{pro2}, as explained in \cite{Andoyo5}. Therefore, for a fixed primitive root $\pi$ of $p$, the graphs described in Proposition~\ref{prop4} (i) are logarithmic cordial graphs to the base $\pi$ modulo $p$.

\end{document}